\documentclass[10pt]{article}

\usepackage{amsmath,amssymb,amsbsy,amsfonts,amsthm,latexsym,amsopn,amstext,
                            amsxtra,euscript,amscd}
\usepackage{booktabs}
\usepackage{tikz}

\numberwithin{equation}{section}
\newtheorem{theorem}{Theorem}[section]

\newtheorem{thm}[theorem]{Theorem}

\newtheorem{lem}[theorem]{Lemma}

\def\\{\cr}
\def\({\left(}
\def\){\right)}
\def\[{\left[}
\def\]{\right]}
\def\<{\langle}
\def\>{\rangle}

\def\N{\mathbb{N}}

\def\Z{\mathbb{Z}}

\def\Q{\mathbb{Q}}

\def\notdivides{\mathrel{\kern-3pt\not\!\kern3.5pt\bigm|}}

\begin{document}


\title{\Large \textbf{Evaluationally coprime linear polynomials}}
\author{%
{\sc RANDELL HEYMAN}  \\
School of Mathematics and Statistics,\\ University of New South Wales \\
Sydney, Australia\\
{\tt randell@unsw.edu.au}
}

\maketitle

\begin{abstract}
Two polynomials, $f,g \in \mathbb{Z}[x]$ are $\emph{evaluationally coprime at x}$ if
\newline
$\gcd(f(x),g(x))=1$. We give necessary and sufficient conditions for two such linear polynomials to have a positive proportion of evaluated coprime values.
\end{abstract}

\newline


\noindent
\section{Introduction}
A natural extension of the greatest common divisor of two polynomials is to consider the greatest common divisor of the evaluation of the two polynomials at a particular value. This then leads to
the concept of polynomials $f,g \in \mathbb{Z}[x]$ that are evaluationally coprime. That is, $\gcd(f(x),g(x))=1$ for all  $x \in \Z$. We can extend this line of enquiry to tuples of evaluationally pairwise coprime polynomials; that is, $f_1,\ldots,f_n$ such that for any $1 \le i< j\le n$ we have
$(f_i(x),f_j(x))=1$ for all $x \in \mathbb{Z}$. 

Denote the greatest common divisor of integers $a_1,\ldots a_n$ by $(a_1,\ldots,a_n)$. Recently, Knox, McDonald and Mitchell \cite{Kno} examined pairs of polynomials $f,g \in \mathbb{Z}[x]$ that have a greatest common divisor equal to 1, and have a greatest common divisor equal to 1 when evaluated at every integer value.
In \cite[Corollary 3.5]{Kno}  necessary and sufficient conditions are given for two primitive linear polynomials to exhibit both of these conditions.
The main result of the present paper, Theorem \ref{main} below, gives necessary and sufficient conditions for the less demanding result that a positive proportion of evaluated values are coprime. Unlike the proof in \cite{Kno}, the proof of Theorem \ref{main} does not use the resultant.
\begin{thm}
\label{main}
Suppose $f(x)=ax+b, ~g(x)=cx+d,\quad  a,b,c,d\in \mathbb{Z},~ a,c \ne 0.$
Then
$$\liminf_{N \rightarrow \infty} \frac{1}{2N+1}\big|\{x : (f(x),g(x))=1, x=-N..N\}\big|>0$$
if, and only if,
$$(a,b,c,d)=1~\textrm{and}~ad \ne bc.$$
\end{thm}
\section{Preparation}
We use the following GCD algorithm (`the algorithm').
Given two polynomials $a_1x+b_1,~ a_2x+b_2 \in \Z[x]$ with $a_1 \ge a_2>0$ we let
\begin{align}
\label{euclid}
a_ix+b_i&=e_{i+1}(a_{i+1}x+b_{i+1})+a_{i+2}x+b_{i+2},~i=1,2,\ldots,
\end{align}
where $e_{i+1}$ is the largest integer such that $e_{i+1}a_{i+1}\le a_i.$
The algorithm terminates when $a_{i+2}=0$. Let $m$ be this value $i+2$. So the algorithm terminates when $a_m=0$.
We note that for any $x\in \Z$ and for any  $1\le i,j\le m-1$ we have
$$(a_ix+b_i,a_{i+1}x+b_{i+1})=(a_jx+b_j,a_{j+1}x+b_{j+1}).$$
We simplify the last part of the algorithm by denoting $a_{m-1}=u,~b_{m-1}=v$ and $b_{m}=s$. So we can write
\begin{align}
\label{abcd=uvs}
(ax+b,cx+d)=(ux+v,s).
\end{align}
To prove Theorem \ref{main}, we require three simple lemmas, below.
\begin{lem}
\label{period}
Let $u,v,s \in \Z$. We have
$(xu+v,s)=((x+s)u+v,s)$ for all $x \in \Z$.
\end{lem}
\begin{proof}
Fix  $x \in \Z$. Let $g_1=(xu+v,s), g_2=((x+s)u+v,s)$.
We have $g_1|su$ so $g_1|(x+s)u+v$; hence $g_1|g_2$.
Similarly, $g_2|su$ so $g_2|xu+v$; hence $g_2|g_1$.
So $g_1=g_2$ as required.
\end{proof}

\begin{lem}
\label{u=(a,c)}
Suppose by comparing the first and last line of the algorithm we have, as shown in \eqref{abcd=uvs},
\begin{align}
\label{abcduvs}
(ax+b,cx+d)=(ux+v,s).
\end{align}
Then $(a,c)=u$ and $(b,d)=(v,s).$
\end{lem}
\begin{proof}
Recalling the algorithm, we have
$$a_ix+b_i=e_{i+1}(a_{i+1}x+b_{i+1})+a_{i+2}x+b_{i+2},~i=1,2,\ldots m-2.$$
Setting $x=0$ and then $x=1$ we have
$$b_i=e_{i+1}b_{i+1}+b_{i+2},\quad a_i+b_i=e_{i+1}(a_{i+1}+b_{i+1})+a_{i+2}+b_{i+2}$$
respectively. Subtracting equations we obtain
$$a_{i}=e_{i+1}a_{i+1}+a_{i+2},$$
where $e_{i+1}$ is the biggest integer such that $e_{i+1}a_{i+1}\le a_{i}$. This is Euclid's algorithm for integers. Thus $(a_{i},a_{i+1})=(a_{i+1},a_{i+2})$.
Since this applies for any $i$ it follows that $(a_1,a_2)=(a_{m-1},0)=a_{m-1}$. Letting $a_1=a,~a_2=c$ and recalling that $a_{m-1}=u$ concludes the proof that $(a,c)=u$.
Setting $x=0$ in \eqref{abcduvs} yields $(b,d)=(v,s)$ which completes the proof.
\end{proof}

\begin{lem}
\label{(a,b,c,d)}
Let $a,b,c,d \in \Z$. We have
$$(a,b,c,d)=((a,b),(c,d)).$$
\end{lem}
\begin{proof}
Let $g_1=(a,b,c,d), g_2=((a,c),(b,d))$.
We have $g_1$ divides both $(a,c)$ and $(b,d)$, so $g_1|g_2$.
Similarly, $g_2|g_1$.
So $g_1=g_2$ as required.
\end{proof}

\section{Proof of theorem}

Suppose $f(x)=ax+b, ~g(x)=cx+d,\quad  a,b,c,d\in \mathbb{Z},~a,c \ne 0.$
Without loss of generality we will assume that $a \ge c$.

To prove sufficiency suppose firstly that $(a,b,c,d)=j\ne 1$. Then for all $x \in \Z$ we have
$j|(ax+b)$ and $j|(cx+d),$
which implies that
$j|(ax+b,cx+d),$ and so
$(ax+b,cx+d)>1.$
Therefore
$$\liminf_{N \rightarrow \infty} \frac{1}{2N+1}\big|\{x : (f(x),g(x))=1, x=-N..N\}\big|=0.$$

Alternately, if $ad = bc$ then, since $a,c \ne 0$, we have $a/c=b/d$ . Thus $a=kc, b=kd$ for some $k \in \Q, k\ge 1$.
So $f(x)=kg(x)$ and the termination line of the algorithm will be
$$(f(x),g(x))=(ux+v,0),$$
for some  $u \in \N, v \in \Z.$

Since $(xu+v,0)=xu+v$ for all $x \in \Z$, the sequence
$$(u+v,0), (2u+v,0),\ldots,$$
is monotonic.  It follows that
$$\liminf_{N \rightarrow \infty} \frac{1}{2N+1}\big|\{x : (f(x),g(x))=1, x=-N..N\}\big|=0.$$

To prove necessity suppose that
$$(a,b,c,d)=1~\textrm{and}~ad \ne bc.$$ Since $ad \ne bc$ then, as argued above,
the right-hand side of the termination line of the algorithm must be
\begin{align}
\label{adbc}
(ux+v,s),~\text{for some}~u \in \Z, s\ne 0.
\end{align}

Using Lemma \ref{period} we see that the sequence
$$(u+v,s),(2u+v,s),\ldots$$ has maximum period $s$. So it will suffice to show that for some $x \in \Z$
we have $(xu+v,s)=1$ for then
$$\liminf_{N \rightarrow \infty} \frac{1}{N}\big|\{x : (f(x),g(x))=1, x=-N..N\}\big|\ge \frac{1}{s}>0.$$

If $(u,s)=1$ then $u^{-1}$ exists modulo $s$. Letting $x=u^{-1}(1-v)$ we obtain $xu+v\equiv 1 \pmod s$ and so, for this value of $x$, we have $(xu+v,s)=1$. So we may assume that $(u,s)=p \ne 1$. If $(u,v)=1$ then, by Dirichlet's theorem on arithmetic progressions \cite{Dir}, there are an infinite number of primes in the arithmetic sequence $\{xu+v\}_{x \in \Z}$. So there must exist a value of $x$ such that $xu+v$ is prime and greater than $s$. It then follows, that for this value of $x$, we have $(xu+v,s)=1$. So we may assume that $(u,v)=q\ne 1$.
Since $q$ divides both $u$ and $v$ we have
$$(xu+v,s)=(q(xuq^{-1}+vq^{-1}),s).$$
Since $(uq^{-1},vq^{-1})=1$ we conclude, using Dirichlet's theorem again, that there must be a value of $x$, denoted $x'$, such that $x'uq^{-1}+vq^{-1}$ is a prime greater than $s$. Letting $r$ be this prime number we have
$$(x'u+v,s)=(qr,s),$$ with $(q,r)=1$.
Since $(a,b,c,d)=1$ we have, using Lemma \ref{(a,b,c,d)}, that $((a,c),(b,d))=1$. We recall, from Lemma \ref{u=(a,c)}, that $(a,c)=u$ and $(b,d)=(v,s)$. Thus $(u,(v,s))=1$. Clearly $(q,s)=1$, for otherwise we would have, using an argument similar to Lemma \ref{(a,b,c,d)},  that $((u,v),s)=(u,(v,s))\ne 1$. Also, $r$ is a prime greater than $s$. So $(r,s)=1$. As $(q,s)=1$ and $(r,s)=1$ it follows that $(qr,s)=1$. So we have
$$(x'u+v,s)=(qr,s)=1,$$ which completes the proof.


\section*{Comments}
There are two lines of enquiry that naturally follow from Theorem \ref{main}.
Firstly, suppose we have (not necessarily linear) integer coefficient  polynomials $f$ and $g$. What are necessary and sufficient coefficient conditions such that
$$\liminf_{N \rightarrow \infty} \frac{1}{N}\big|\{x : (f(x),g(x))=1, x=-N..N\}\big| >0?$$

Secondly, suppose we have linear integer coefficient polynomials, $f_1,\ldots,f_n$.  What are necessary and sufficient coefficient conditions such that
$$\liminf_{N \rightarrow \infty} \frac{1}{N}\big|\{x : (f_1(x),\ldots,f_n(x))=1, x=-N..N\}\big| >0?$$

\section*{Acknowledgements}
The author thanks Adrian Dudek for posing the question that resulted in Theorem \ref{main}, in a project unrelated to \cite{Kno}.
The author thanks Gerry Myerson for some views on results in this area and Thomas Britz for some useful comments.

\makeatletter
\renewcommand{\@biblabel}[1]{[#1]\hfill}
\makeatother

\end{document}